\documentclass[graybox]{svmult} 
\usepackage{mathptmx}       
\usepackage{helvet}         
\usepackage{courier} 

\usepackage{makeidx}         
\usepackage{graphicx}        
\usepackage{multicol}        
\usepackage[bottom]{footmisc}
\usepackage{amsmath,amssymb,amsfonts}        
\makeindex     

\usepackage{slashed}
\usepackage[T1]{fontenc}
\usepackage[utf8]{inputenc}
\usepackage{hyperref,graphicx,multicol,enumerate} 
\usepackage[bottom]{footmisc}
\usepackage{textgreek,bm,upgreek,mathrsfs,stmaryrd,color,epigraph}
\usepackage[many]{tcolorbox}
\usepackage[english]{babel}

\usepackage{pgfplots}
\usepgfplotslibrary{fillbetween}

\usepackage{dsfont}

\newtheorem{thm}{Theorem}
\newtheorem{lem}[thm]{Lemma}
\newtheorem{rmk}[thm]{Remark}
\newtheorem{crl}[thm]{Corollary}

\newcommand{\R}{\mathbb{R}}

\newcommand{\eps}{\varepsilon}
\newcommand{\norm}[1]{\left\lVert#1\right\rVert}
\newcommand{\eq}[1]{\begin{equation}\begin{aligned} #1 \end{aligned}\end{equation}}
\newcommand{\fm}[1]{\[\begin{aligned} #1 \end{aligned}\]}
\newcommand{\lra}[1]{\langle #1 \rangle}

\newcommand{\ud}{\mathrm{d}}

\begin{document}

\title*{John's blow up examples and scattering solutions for semi-linear wave equations}
\author{Louie Bernhardt, Volker Schlue, Dongxiao Yu}
\institute{Louie Bernhardt \at School of Mathematics and Statistics, University of Melbourne, Australia\\\email{lbernhardt@student.unimelb.edu.au}
\and
Volker Schlue \at School of Mathematics and Statistics, University of Melbourne, Australia\\\email{volker.schlue@unimelb.edu.au}
\and
Dongxiao Yu \at Department of Mathematics, University of California, Berkeley, US\\\email{yudx@math.berkeley.edu}}

\maketitle

\abstract{
In light of recent work of the third author, we revisit a classic example given by Fritz John of a semi-linear wave equation which exhibits finite in time blow up for all compactly supported data.
We present the construction of future global solutions from asymptotic data given in \cite{yu2022nontrivial} for this specific example,
and clarify the relation of this result of Yu to John's theorem in \cite{MR0600571}.
Furthermore we present a novel blow up result for finite energy solutions satisfying a sign condition due to the first author,
and invoke this result to show that the constructed backwards in time solutions blow up in the past.
}

\section{Introduction}

At the MATRIX workshop on \textbf{Hyperbolic PDEs and Nonlinear Evolution Problems}, in September 2023, Dongxiao Yu presented several global existence results \textit{from asymptotic data} for a class of quasi-linear wave equations \cite{yu2022nontrivial}. These \textit{backwards in time} results are compelling because they apply in particular to wave equations for which \textit{forward in time} global existence is not known, or worse, are known to admit blow up solutions.

The topic of this note is the classic example 
\begin{equation}
\label{eq:john:intro}
    -\Box u=(\partial_t u)^2\qquad \text{: on } \mathbb{R}^{3+1}\,.
\end{equation}
This is a \textit{semi-linear} wave equation for which Fritz John showed that \emph{forward in time solutions arising from data of compact support} blow up in finite time \cite{MR0600571}. We recall the statement of this well-known fact, and his argument briefly in Section~\ref{sec:forward}; excellent expositions of John's results in \cite{MR0600571} can be found in the lecture notes \cite{notes:const,notes:holzegel}.

The forward in time results use the initial value formulation for \eqref{eq:john:intro},
and rely on the method of spherical means.
The \emph{backward in time solutions} constructed by Yu
use \textit{H\"ormander's asymptotic system} \cite{MR897781,MR1120284,MR1466700} to find an approximate solution,
and rely on energy estimates for the remainder.

A similar philosophy has been applied by Lindblad and Schlue \cite{LS}
to obtain global existence from \textit{scattering data, backwards in time},
for \textit{semi-linear} wave equations which admit \textit{forward in time} global solutions.
The classic examples are here wave equations satisfying the null condition \cite{K86},
\begin{equation}
    \label{eq:null}
    \Box u=Q(\partial u,\partial u)\,
\end{equation}
which, in contrast to \eqref{eq:john:intro} admit global solutions which asymptote to a given \textit{linear} solution, and 
\begin{equation}
\label{eq:lin}
    u(t,x)\sim \varepsilon r^{-1} U(q,\omega) \qquad (r\sim t, t\to\infty)  
\end{equation}
where $r=|x|$, $q=r-t$, and $\omega=x/|x|$.
In particular, scattering solutions can be constructed from knowledge of the \textit{radiation field} $U(q,\omega)$.

The situation is very different for \eqref{eq:john:intro}.
As we will discuss in this note,  the existence results for \eqref{eq:john:intro} that go beyond \textit{almost global existence} for small data, depend on carefully chosen data, either asymptotically or initially.

In Section~\ref{sec:backward},
we demonstrate that there are solutions which  approach a solution of the asymptotic system,
in the sense that
\begin{equation}
    u(t,x)\sim \varepsilon r^{-1} U(s,q,\omega) \qquad (r\sim t, t\to\infty)  
\end{equation}
where $s=\varepsilon \ln(t)$, and $U$ satisfies
\begin{equation}
    \label{eq:asymp:intro}
    2U_{sq}+U_q^2=0.
\end{equation}
Following \cite{yu2022nontrivial},
we first prepare a solution to \eqref{eq:asymp:intro} which does not blow up,
and then find a corresponding solution to \eqref{eq:john:intro}.
This construction is a special case of the results in \cite{yu2022nontrivial} which apply more generally to \textit{quasi-linear wave equations}. In contrast to \eqref{eq:lin}, these solutions \textit{do not have a radiation field}.

The main result that motivates this note is Theorem~\ref{thm:yu}, in Section~\ref{sec:backward}:
There do exist $\mathrm{C}^k$ solutions of \eqref{eq:john:intro} for $t\geq 0$.
We will call these solutions \textbf{future global}.

In view of John's blow up result  
these solutions cannot be compactly supported at $t=0$.
Indeed, while Theorem~\ref{thm:john} states that the only \textbf{global} $C^k$  solution that arises from compactly supported data is the trivial solution,
it is clear from John's argument in Section~\ref{sec:forward}  that solutions of \eqref{eq:john:intro} arising from compactly supported data \textit{blow up both to the future, and to the past}. (They are neither future nor past global.)

This also shows that the blow-up results of John \cite{MR0600571} do not trivially extend from compactly supported data to general finite energy solutions. As we have seen, there are \textit{some} finite energy solutions, namely the solutions of Theorem~\ref{thm:yu} which are future global.
However, this raises the question if these solutions are also \textbf{past global}.

In Section~\ref{sec:sign},
we  show with an argument that is due to the first author,
that the solutions of Theorem~\ref{thm:yu} are not \textbf{global}:
While they are future global, they \textit{blow up in the past}.
Theorem~\ref{thm:bernhardt} gives a sufficient sign condition for blow up of finite energy solutions to \eqref{eq:john:intro}, and Corollary~\ref{crl:thm2solblowup} shows that the solutions constructed in \cite{yu2022nontrivial} satisfy that sign condition.

The occurence of a sign condition is expected in this setting:
Already for the asympototic equation \eqref{eq:asymp:intro} the sign of the initial data for $U_q$ at $s=0$ is relevant for existence, for $s>0$.

Thus this note clarifies the relation between the global existence results of the third author,
and the blow up results of John, in the special case of the equation \eqref{eq:john:intro}.
 We leave open the question if there are \textit{any} global finite energy solutions to \eqref{eq:john:intro}.

\section{John's blow-up result for the forward solution}
\label{sec:forward}

Fritz John's classical paper \cite{MR0600571} contains the result that is most relevant to this note:

\begin{thm}[John '81]
\label{thm:john}
  Suppose $u:\mathds{R}^{3+1}\to\mathds{R}$ is a solution of
  \begin{gather}
    \label{eq:12}
    \partial_t^2u-\Delta u=(\partial_t u)^2\,,\\
    u(0,x)=u_0\,,\quad \partial_tu(0,x)=u_1
  \end{gather}
  with \emph{smooth data of compact support} $u_0,u_1\in\mathrm{C}^\infty_0$.
  
  If $u\in\mathrm{C}^2(\mathds{R}^3\times [0,\infty)$, then  $u=0$ identically. 
\end{thm}

\begin{rmk}
    \label{rmk:john:example}

The theorem refers to \emph{general solutions} with compactly supported data.
The more frequently referred to statement, which is also easier to prove, is the fact that there are \emph{some} solutions that blow up.
These are constructed from \emph{spatially homogeneous} solutions, $u(t,x)=u(t)$, which satisfy the ODE
\begin{equation}
\label{eq:john:example}
    \partial_t u_t=(u_t)^2\,.
\end{equation}
This ODE should be compared to \eqref{eq:asymp:intro}.
In both cases the \emph{sign} of the initial data is crucial:
In \eqref{eq:john:example} it is chosen to ensure blow-up,
and for \eqref{eq:asymp:intro} it is chosen to \emph{prevent} blow-up.
In view of these simple examples it is also not surprising that a \emph{sign condition} is relevant in the general setting of Section~\ref{sec:sign}.

\end{rmk}

The proof of Theorem~\ref{thm:john} is a spherical means argument,
and relies on the fact that in  $3+1$ dimensions
the spherical mean
\begin{equation}
  M_{t,r}[u](x)=\frac{1}{4\pi}\int_{|\xi|=1}u(t,x+r\xi)\ud \sigma(\xi)
\end{equation}
of any solution to \eqref{eq:12} satisfies
\begin{equation}
  \label{eq:15}
  \partial_t^2 M_{t,r}[ru](x)-\partial_r^2 M_{t,r}[ru](x)=M_{t,r}[r(\partial_t u)^2](x)\,.
\end{equation}
This well-known reduction to a $1+1$-dimensional wave equation in $(t,r)$
then yields the representation formula
\begin{multline}
  \label{eq:mean:dalembert}
  M_{t,r}[ru](0)=\frac{1}{2}\bigl((r+t) M_{r+t}[u_0](0)+(r-t)M_{r-t}[u_0](0)\bigr)+\frac{1}{2}\int_{r-t}^{r+t} \rho M_\rho [u_1](0)\ud \rho\\
  +\frac{1}{2}\int_0^t \int_{r-(t-s)}^{r+(t-s)} \rho M_{s,\rho}[(\partial_t u)^2](0)\ud\rho  \ud s
\end{multline}
\emph{For data of compact support}, say contained in the ball $|r|\leq R$,
and any $(t,r)$ with $t>R+|r|$, $r>0$,
it then holds
\begin{equation}
  \label{eq:mean:lower}
  M_{t,r}[ru](0)\geq \frac{1}{2}\int_0^t \int_{|t-r-s|}^{r+t-s} \rho M_{s,\rho}^2[\partial_t u](0)\ud\rho  \ud s \,.
\end{equation}
Given that the integrand is non-negative it can be suitably restricted further.
Most importantly it is an integral in space \emph{and time},
over the spherical means of $\partial_t u$;
since by the compact support assumption $u$ vanishes for $r>R+t$,
the latter integral in time can be carried out,
and by the fundamental theorem of Calculus and Cauchy-Schwarz this leads to 
\begin{equation}
  \label{eq:mean:lower:calc}
  M_{t,r}[u](0)\geq \frac{1}{2} \int_{t-r}^r \tfrac{\rho}{r(R+t-r)} \bigl|M_{\rho+(t-r),\rho}[u](0)\bigr|^2 \ud\rho\,,
\end{equation}
for $(t,r)\in \Sigma_R=\{(t,r): R<R+r<t<2r\}$.
For a detailed derivation of this inequality we refer the reader to \cite{notes:const, notes:holzegel}.

This now shows that \emph{along a characteristic emanating} from $(t,r)$,
the quantity
\begin{equation}
  \label{eq:25}
  \beta(\rho)=\int_{-q}^{\rho}\sigma \bigl|M_{\sigma-q,\sigma}[u]\bigr|^2 \ud\sigma
\end{equation}
satisfies the differential inequality
\begin{equation}
  \label{eq:26}
  \beta'(\rho)=\rho \bigl|M_{\rho-q,\rho}[u]\bigr|^2\geq \frac{1}{4(R-q)^2\rho}\beta^2(\rho)\,.
\end{equation}
It is then easy to show that this quantity blows up, unless $M_{\rho-q,\rho}[u]=0$ for all $\rho\geq r$.

Note that the blow up here is at the level of the solution $u$ itself.
This is distinct from the blow-up argument in Section~\ref{sec:sign},
which yields blow up at the level of $\partial_t u$.

\section{Existence of the backwards solution}
\label{sec:backward}

Consider the scalar semilinear wave equation \eqref{eq:john:intro} in $\R^{1+3}$:
\fm{-\Box u=(\partial_t^2-\Delta) u=(\partial_tu)^2.}
We have seen in the previous section,
that according to the result of John \cite{MR0600571},
any nontrivial $C^3$ solution to \eqref{eq:john:intro} with compactly supported initial data blows up in finite time. In this section, however, we discuss a recent result in  \cite[Corollary 8.3]{yu2022nontrivial} on the construction of nontrivial future global solutions to \eqref{eq:john:intro}.
\begin{thm}[Yu '23] \label{thm:yu}
There exist a large family of nontrivial future global solutions to \eqref{eq:john:intro} for $t\geq 0$. Moreover, for each of these future global solutions, the solution and its time derivative cannot be both compactly supported at $t=0$.
\end{thm}\rm

\rmk{\rm For each constant $C>0$, the function $-\ln(t+C)$ is a future global solution to \eqref{eq:john:intro}; cf.~Remark~\ref{rmk:john:example}. However, the future global solutions constructed in \cite{yu2022nontrivial} decay as $|x|\to\infty$ and thus are different from these constant-in-$x$ solutions. See, e.g., the estimates \eqref{eq:intder} below.}\rm

\rmk{\rm \label{thm:yu:decaya}To be more precise, in this paper, we prove that for each nonzero function $A(q,\omega)\in C_c^\infty(\R\times\mathbb{S}^2)$ with $A\geq 0$, there exists a nontrivial future global solution $u$ to \eqref{eq:john:intro}. The solution $u$ is associated to the function $A$ in the following way. At infinite time, the solution $u$ matches a future global solution $U(s,q,\omega)$ to the asymptotic equation \eqref{eqn:hormasy}. Meanwhile, this $U$ is uniquely determined by the data $U_q|_{s=0}=A$ and the boundary condition $\lim_{q\to-\infty}U(s,q,\omega)=0$.}\rm

\rmk{\rm Note that the second half of Theorem \ref{thm:yu} follows directly Theorem \ref{thm:john}. Meanwhile, we can also prove that the data of these nontrivial future global solutions are not compactly supported without applying John's finite time blowup result. See Section \ref{sec:yu:noncompact}.}\rm

\rmk{\rm\label{thm:yu:qlw} Theorem \ref{thm:yu} can be generalized. First, it was mentioned in Remark \ref{thm:yu:decaya} that we start with some function $A$ and construct a corresponding future global solution to \eqref{eq:john:intro}. The assumptions of $A$ can be relaxed. Instead of assuming $A\in C_c^\infty$, we only need to assume that $A$ along with its derivatives decays faster than certain negative powers of $\lra{q}$ as $|q|\to\infty$. 

Besides, the construction used in the proof of Theorem \ref{thm:yu} also applies to a large family of quasilinear wave equations in $\R^{1+3}$. 
We refer our readers to \cite{yu2022nontrivial} for more details. }
\rm



\subsection{H\"ormander's asymptotic equation}\label{sec:yu:asyeqn}
Let us discuss how the construction works. Here we make use of an asymptotic equation for \eqref{eq:john:intro} which was first introduced by H\"ormander \cite{MR897781,MR1120284,MR1466700}. In its derivation, we assume that we already have a future global solution $u=u(t,x)$ for $t\geq 0$ to \eqref{eq:john:intro} and that $u$ has a special form
\eq{u(t,x)=\eps r^{-1}U(s,q,\omega)}
where $s=\eps\ln t$, $r=|x|$, $\omega=x/|x|$, and $q=r-t$. Assuming that $t=r\to\infty$, we plug $u=\eps r^{-1}U$ into \eqref{eq:john:intro} and compare the coefficients of terms of order $\eps^2 t^{-2}$. This gives us the following  asymptotic equation for\eqref{eq:john:intro}:
\eq{\label{eqn:hormasy}2U_{sq}+U_q^2=0.}
We can solve this equation explicitly:
\eq{\label{eqn:hormasy:sol} U_q(s,q,\omega)&=\frac{2U_q(0,q,\omega)}{U_q(0,q,\omega)s+2}} 

\begin{rmk}
If  $U|_{s=0}\in C_c^\infty$ is nonzero, we have $U_q(0,q,\omega)<0$ for some $(q,\omega)$, so the denominator in \eqref{eqn:hormasy:sol} equals $0$ for some positive $s$. In this case, we have finite time blowup for the asymptotic equation \eqref{eqn:hormasy} when the data are $C_c^\infty$ and nonzero. 

We may compare this observation to Remark~\ref{rmk:john:example}.
Indeed the asymptotic equation \eqref{eqn:hormasy:sol} is precisely of the form \eqref{eq:john:example}.
In both cases the blow-up is sensitive to the sign of the initial data.

\end{rmk}

\subsection{Construction of nontrivial future global solution}\label{sec:yu:construction}
In our construction, we want to obtain a future global solution to \eqref{eqn:hormasy} for $s\geq 0$. To achieve this goal, we fix a function $A=A(q,\omega)$ such that $A\geq 0$ everywhere. For simplicity, in this note we also assume that $A\in C_c^\infty$. By setting $U_q(0,q,\omega)=A(q,\omega)$ in \eqref{eqn:hormasy:sol} and assuming $\lim_{q\to-\infty}U(s,q,\omega)=0$, we obtain a uniquely determined future global solution $U$ to \eqref{eqn:hormasy} defined by
\eq{\label{eqn:hormasy:sol2} U(s,q,\omega)&=\int_{-\infty}^q\frac{2A(\rho,\omega)}{A(\rho,\omega)s+2}\ d\rho,\qquad \forall (s,q,\omega)\in[0,\infty)\times\R\times\mathbb{S}^2.}
Since $A\in C_c^\infty$, if we choose $R>0$ so that $A\equiv 0$ for $|q|\geq R$, then we have $U\equiv 0$ whenever $q\leq -R$. 
Using this $U$, we obtain an approximate solution $u_{app}=u_{app}(t,x)$ as follows. Fix a small constant $\delta\in(0,1)$ and set
\eq{u_{app}(t,x)&=\eps r^{-1}\eta(t)\psi(r/t)U(\eps\ln t-\delta,r-t,\omega).}
Here $\eta=\eta(t)$ and $\psi=\psi(s)$ are two cutoff functions. We choose them so that $u_{app}\equiv \eps r^{-1}U$ in the region where $t\gtrsim 1$ and $|r-t|<t/4$ and that $u_{app}\equiv 0$ whenever $t\lesssim 1$ or $|r-t|>t/2$. And since $U$ vanishes for $q\leq -R$, we have $u_{app}\equiv 0$ whenever $r-t\leq -R$. In \cite[Section 4]{yu2022nontrivial},  it is shown that with this choice of  $u_{app}$,  we have for $\eps\ll1$ the pointwise estimates
\fm{|Z^Iu_{app}|\lesssim_I\eps (1+t)^{-1+C_I\eps} ,\qquad\forall I,\ \forall (t,x)\in[0,\infty)\times\R^3}
and $u_{app}$ is indeed an approximate solution to \eqref{eq:john:intro} in the following sense:
\fm{|Z^I(\Box u_{app}+(\partial_tu_{app})^2)|\lesssim_I\eps (1+t)^{-3+C_I\eps},\qquad \forall I,\  \forall (t,x)\in[0,\infty)\times\R^3.}
Here $Z^I$ denotes a product of $|I|$  commuting vector fields: translations $\partial_\alpha$ for $\alpha=0,1,2,3$, scaling $S=t\partial_t+r\partial_r$, rotations $\Omega_{ij}=x_i\partial_j-x_j\partial_i$ for $i,j=1,2,3$ and Lorentz boosts $\Omega_{0i}=x_i\partial_t+t\partial_i$ for $i=1,2,3$.

Next, we fix a large time $T>1$ and consider the following backward Cauchy problem
\eq{\label{backwardcauchy}-\Box v=(v_t+2\partial_tu_{app})v_t-\chi(t/T)(-\Box u_{app}-(\partial_tu_{app})^2)\qquad\text{ in }[0,\infty)\times\mathbb{R}^3}
with $v\equiv 0$ for $t\geq 2T$. Here $\chi\in C_c^\infty(-2,2)$ is a cutoff function such that $0\leq \chi\leq 1$ and $
\chi|_{[-1,1]}\equiv 1$. It is not hard to see that $u:=v+u_{app}$ solves \eqref{eq:john:intro} for $t\leq T$ and that $u^T\equiv u_{app}$ for $t\geq 2T$. In order to solve \eqref{backwardcauchy}, we make use of a continuity argument. Here, since the equation \eqref{eq:john:intro} is semilinear and since we choose $A\in C_c^\infty$, in our proof, we can use the unweighted energy $E(v)(t):=\norm{\partial v(t)}^2_{L^2(\R^3)}$ and the standard energy estimate for $\Box$. \footnote{In Remark \ref{thm:yu:qlw}, it was mentioned that nontrivial future global solutions can also be constructed under weaker assumptions on $A$ or for general quasilinear wave equations. To achieve these goals, we define new weighted energies and apply weighted energy estimates and Poincar$\acute{\rm e}$'s estimates.  All of these are motivated by, for example,  \cite{MR2382144,MR2134337,MR2680391}. We also refer \cite{yu2022nontrivial} for more details.} The details of the proof can be found in \cite[Section 6]{yu2022nontrivial}.
Finally, we show that the limit $v^\infty=\lim_{T\to\infty} v^T$ exists. To achieve this goal, we show that $\{v^T\}_{T>1}$ is a Cauchy sequence in suitable Banach spaces, which are related to the energy estimates used in this setting. We refer to \cite[Section 7]{yu2022nontrivial} for the details of the proof. In summary, for a fixed large integer $N$ and for all $\eps\ll_N1$, we obtain a future global $C^{N+1}$ solution  $u=u^\infty(t,x):=v^\infty(t,x)+u_{app}(t,x)$ to \eqref{eq:john:intro} for all $t\geq 0$ such that $u|_{r-t\leq -R}\equiv 0$ where $R>0$ is a constant such that $A\equiv $ for $|q|\geq R$.
Moreover, we have the following estimate in energy for the remainder
\eq{\label{est:uuappdiff}\sum_{|I|\leq N-1}\norm{\partial Z^I(u-u_{app})(t)}_{L^2(\R^3)}\lesssim_N\eps (1+t)^{-1/2+C_N\eps},\qquad \forall t\geq 0.}
We emphasize that these estimates imply in particular that the scattering solutions have finite energy on each time slice. This is because $u_{app}(t)$ is compactly supported for each fixed time $t\geq 0$. 
The estimates \eqref{est:uuappdiff} also imply using the Klainerman-Sobolev inequality the pointwise estimates
\eq{\label{eq:intder}\sum_{|I|\leq N-1}|\partial Z^I(u-u_{app})(t,x)|\lesssim_N\eps (1+t)^{-1/2+C_N\eps}(1+t+|x|)^{-1}(1+|t-|x||)^{-1/2},\qquad \forall t\geq 0.}
By integrating these pointwise estimates and noticing that $u|_{r-t\leq -R}\equiv 0$, we also obtain pointwise bounds for $Z^I(u-u_{app})$.

\subsection{Proof that the data are not compactly supported}\label{sec:yu:noncompact}

We finally remark that the future global solution $u$ constructed above cannot have compactly supported initial data at $t=0$ as long as $A\not\equiv 0$. In fact, the pointwise bound for $\partial (u-u_{app})$ above shows that
\eq{\label{eq:int}|u-u_{app}|\lesssim \eps (1+t)^{-3/2+C\eps}(1+|t-|x||)^{1/2},\qquad \text{whenever }|x|<2t.}
Recall that $u_{app}=\eps r^{-1}U$ in the region $|r-t|\leq t/4$, so for each $\gamma\in(0,1)$,  we have
\fm{|u(t,x)-\eps r^{-1}U(\eps\ln t-\delta,r-t,\omega)|\lesssim \eps t^{-(3+\gamma)/2+C\eps},\qquad \forall t\geq 1,\ |t-|x||\leq t^{\gamma}.}
The power of $t$ is less than $-1$ for $\eps\ll_\gamma1$. Meanwhile, if $A(q^0,\omega^0)>0$ for some $(q^0,\omega^0)$, we have for all $q>q^0$,
\fm{U(s,q,\omega^0)&=\int_{-\infty}^q\frac{2A(\rho,\omega^0)}{A(\rho,\omega^0)s+2}\ d\rho\geq \int_{-\infty}^{q^0}\frac{2A(\rho,\omega^0)}{A(\rho,\omega^0)s+2}\ d\rho\gtrsim_{q^0,\omega^0} (1+s)^{-1}.}
As a result, for all $t\gtrsim 1$, $r-t>q^0$, and $|r-t|\leq t^\gamma$ with $\gamma\in(0,2/3)$, we have
\fm{u(t,r\omega^0)\geq C_{q^0,\omega^0}^{-1}\eps r^{-1}\cdot (1+\eps\ln t)^{-1}-C\eps t^{-(3+\gamma)/2+C\eps}\gtrsim_{q^0,\omega^0,\gamma} \frac{\eps}{t(\eps\ln t+1)}.}

However, if the data of $u$ is compactly supported, we have $u\equiv 0$ whenever $r-t\geq R_0$ for some constant $R_0>0$. Note that the intersection of $\{r-t\geq R_0\}$ and $\{t>1,r-t>q^0,|r-t|<t^\gamma\}$ is not empty for all sufficiently large $t$, so here is a contradiction. We conclude that the solution $u$ constructed above cannot have compactly supported initial data.

\section{Blowup criteria for solutions with non-compactly supported initial data}\label{sec:sign}
In this section we give sufficient conditions for the blow up in finite time of solutions to the semilinear wave equation
\begin{equation}\label{eq:wavejohn}
\square u = -(\partial_t u)^2.
\end{equation}
These conditions amount to assumptions on the sign of the solution to the linear wave equation with the same initial data, as well as its time derivative. Unlike previous proofs of blowup for solutions to \eqref{eq:wavejohn}, including John's original proof in \cite{MR0600571}, we do not assume that the data is compactly supported. In fact, Theorem \ref{thm:bernhardt} shows that a large class of finite-energy solutions blow up in finite time. We note that the sign conditions that imply blowup, namely \eqref{eq:signconditionforward}, \eqref{eq:signconditionbackward} in Lemma \ref{lem:signconditions}, constitute sufficient conditions for solutions to \eqref{eq:wavejohn} to blow up in finite positive time or finite negative time respectively.

First we show that, under suitable decay assumptions on our initial data, solutions to the linear equation satisfy one of three sign conditions. The proof of Lemma \ref{lem:signconditions} (and of Theorem \ref{thm:bernhardt}) rely heavily on taking spherical means of solutions, and so we introduce the following notation. Given a function $f\in C(\R^3)$, we denote the spherical mean of radius $r \geq 0$ of $f$ about the point $x_0 \in \R^3$ as
\begin{equation}\label{eq:sphmeandef}
M_{x_0}[f](r) = \frac{1}{4\pi}\int_{|\xi|=1}f(x_0+r\xi)\ud \sigma(\xi).
\end{equation}

\begin{lem}\label{lem:signconditions}
Let $u_0,u_1 \in C^2(\R^3)$ such that
\begin{equation}\label{eq:datadecay}
\lim_{r\rightarrow\infty}\partial_r(ru_0) = \lim_{r\rightarrow \infty} ru_1 = 0,
\end{equation}
where $r = |x|$. Assume one of $u_0,u_1$ are nonzero. Let $u_{\text{lin}}$ be the unique solution to the linear wave equation
\begin{equation*}
\square u_{\text{lin}} = 0
\end{equation*}
with initial data $u_{\text{lin}}(0,x) = u_0(x)$, $\partial_t u_{\text{lin}}(0,x) = u_1(x)$. Then $u_{\text{lin}}$ satisfies at least one of the following three sign conditions:
\begin{enumerate}
\item There exists $x_0 \in \R^3$, $q \geq 0$, $r_0 \geq q$, such that
\begin{equation}\label{eq:signconditionforward}
M_{x_0}[\partial_t u_{\text{lin}}(r-q)](r) > 0
\end{equation}
for all $r \geq r_0$.
\item There exists $x_0 \in \R^3$, $q \geq 0$, $r_0 \geq q$ such that
\begin{equation}\label{eq:signconditionbackward}
M_{x_0}[\partial_t u_{\text{lin}}(-(r-q))](r) < 0,
\end{equation}
for all $r \geq r_0$.
\item For all $t\in \R,x \in \R^3$, we have
\begin{equation}\label{eq:signcondition3}
u_{\text{lin}}(t,x) \geq 0.
\end{equation}
\end{enumerate}
\end{lem}

\begin{proof}
We write explicitly
\begin{multline*}
M_{x_0}[u_{\text{lin}}(t)](r) = \frac{1}{2r}\Big((r+t)M_{x_0}[u_0](r+t) + (r-t)M_{x_0}[u_0](r-t)\Big)\\
+ \frac{1}{2r}\int_{r-t}^{r+t}\rho M_{x_0}[u_1](\rho)\ud\rho.
\end{multline*}
Differentiating in $t$, we have
\begin{multline*}
M_{x_0}[\partial_t u_{\text{lin}}(t)](r) = \frac{1}{2r}\Big(\partial_r(rM_{x_0}[u_0])(r+t) -\partial_r(rM_{x_0}[u_0])(r-t)\Big)\\
+ \frac{1}{2r}\Big( (r+t)M_{x_0}[u_1](r+t) + (r-t)M_{x_0}[u_1](r-t)\Big).
\end{multline*}
The right hand side can be written entirely in terms of the solution $u_{\text{lin}}$ as
\begin{equation}\label{eq:sphericalmeanssol}
M_{x_0}[\partial_t u_{\text{lin}}(t)](r) = \frac{1}{2r}\Big(u_{\text{lin}}(r+t,x_0) - u_{\text{lin}}(-(r-t),x_0)\Big).
\end{equation}
One may see this from Kirchhoff's formula for $u_{\text{lin}}$:
\begin{align*}u_{\text{lin}}(t,x_0) &= \frac{1}{4\pi t^2}\int_{\partial B_t(x_0)} \partial_r (ru_0)(t,y) + tu_1(t,y) \ud \sigma(y)\\
&= \frac{1}{4\pi}\int_{|\xi| = 1} \partial_r (ru_0)(t,x_0+t\xi) + tu_1(t,x_0+t\xi) \ud \sigma(\xi)\\
&= \partial_r (rM_{x_0}[u_0])(t) + tM_{x_0}[u_1](t),
\end{align*}
where the radial coordinate $r$ is centered at $x_0$. Letting $q = r-t$, we write \eqref{eq:sphericalmeanssol} as
\begin{equation*}
M_{x_0}[\partial_t u_{\text{lin}}(r-q)](r) = \frac{1}{2r}\Big(u_{\text{lin}}(2r-q,x_0) - u_{\text{lin}}(-q,x_0)\Big).
\end{equation*}
We now claim that the first sign condition \eqref{eq:signconditionforward} is satisfied if there exists some $x_0 \in \R^3$, $q \geq 0$ such that $u_{\text{lin}}(-q,x_0) < 0$. The decay condition \eqref{eq:datadecay} implies $u_{\text{lin}}(2r-q) \rightarrow 0$ as $r \rightarrow \infty$, since
\begin{equation*}
u_{\text{lin}}(2r-q,x_0) = \partial_r(rM_{x_0}[u_0])(2r-q) + (2r-q)M_{x_0}[u_1](2r-q).
\end{equation*}
and so if $\partial_r(ru_0),ru_1 \rightarrow 0$ as $r \rightarrow \infty$, so do their spherical means, therefore there exists some $r_0 \geq q$ such that 
\begin{equation*}
|u_{\text{lin}}(2r-q,x_0)| < \frac{1}{2}|u_{\text{lin}}(-q,x_0)|
\end{equation*}
for all $r \geq r_0$. It follows that
$$M_{x_0}[\partial_t u_{\text{lin}}(r-q)](r) > -\frac{1}{2}u_{\text{lin}}(-q,x_0) > 0$$
for all $r \geq r_0$, and thus \eqref{eq:signconditionforward} is satisfied. Similarly, from \eqref{eq:sphericalmeanssol} we have
\begin{equation*}
M_{x_0}[\partial_t u_{\text{lin}}(-(r-q))](r) = \frac{1}{2r}\Big(u_{\text{lin}}(q,x_0) - u_{\text{lin}}(-(2r - q),x_0)\Big),
\end{equation*}
and so by the same argument the second sign condition \eqref{eq:signconditionbackward} is satisfied if there exists some $x_0 \in \R^3$, $q \geq 0$ such that $u_{\text{lin}}(q,x_0) < 0$. It follows that if neither of \eqref{eq:signconditionforward},\eqref{eq:signconditionbackward} hold, then the third sign condition \eqref{eq:signcondition3} is necessarily satisfied.
\end{proof}

\rmk{\rm \label{rmk:signconditionnontrivial}We point out that the decay condition \eqref{eq:datadecay} is satisfied by all continuous differentiable initial data with finite energy. Moreover, for each sign condition, there are finite-energy solutions to the homogeneous wave equation that satisfy that one sign condition and not the other two. For example, consider the solution $u_{\text{lin}}$ with initial data
$$u_{\text{lin}}(0,x) = 0,\quad \partial_t u_{\text{lin}}(0,x) = \chi(x),$$
where $\chi \in C_0^\infty(\R^3)$ is some nonnegative bump function. Then for all $t \geq 0,x\in\R^3$, we have
$$u_{\text{lin}}(-t,x) = -tM_{x}[\chi](t),$$
and
$$\partial_t u_{\text{lin}}(t,x) = \partial_t u_{\text{lin}}(-t,x) =  M_{x}[\chi](t).$$
Therefore $\partial_t u_{\text{lin}} \geq 0$ everywhere and $u_{\text{lin}} \leq 0$ for all negative $t$, and so $u_{\text{lin}}$ satisfies the first sign condition \eqref{eq:signconditionforward}, and not the other two. Similarly, the solution $v_{\text{lin}}(t,x) = -u_{\text{lin}}(t,x)$ satisfies the second sign condition \eqref{eq:signconditionbackward} only. For the final sign condition, we consider the solution $u_{\text{lin}}$ with initial data
$$u_{\text{lin}}(0,x) = \frac{1}{(1+|x|^2)^{1/2}},\quad \partial_t u_{\text{lin}}(0,x) = 0.$$
Then we have
$$u_{\text{lin}}(t,x) = \frac{1}{2|x|}\Big(\frac{|x|+t}{(1+|x+t|^2)^{1/2}} + \frac{|x|-t}{(1+|x-t|^2)^{1/2}}\Big).$$
Then $u_{\text{lin}} \geq 0$ everywhere, and one can show that for all $t \geq 0$, we have
$$\partial_t u_{\text{lin}}(t,x) \leq 0 \leq \partial_t u_{\text{lin}}(-t,x).$$
Hence $u_{\text{lin}}$ satisfies only the third sign condition \eqref{eq:signcondition3}.}\rm

We now show that the first two sign conditions \eqref{eq:signconditionforward}, \eqref{eq:signconditionbackward} are sufficient conditions for finite-time blowup.
\begin{thm}[Bernhardt '24]
\label{thm:bernhardt}
Suppose $u$ is a nonzero solution of the initial value problem \eqref{eq:12} with initial data $u(0,x) = u_0(x)$, $\partial_t u(0,x) = u_1(x)$. Suppose $u_0,u_1 \in C^2(\R^{3})$ satisfy the decay assumption \eqref{eq:datadecay}, and let $u_{\text{lin}}$ be the solution to the homogeneous wave equation $\square u_{\text{lin}} = 0$ with matching initial data $u_0,u_1$. If $u_{\text{lin}}$ satisfies the sign condition \eqref{eq:signconditionforward}, then $\partial_t u \rightarrow \infty$ in finite positive time. Alternatively, if $u_{\text{lin}}$ satisfies \eqref{eq:signconditionbackward}, then $\partial_t u \rightarrow -\infty$ in finite negative time.
\end{thm}

\begin{proof}
First, we assume that the first sign condition \eqref{eq:signconditionforward} holds for some $x_0 \in \R^3, q \geq 0, r_0 \geq q$. The spherical mean $M_{x_0}[u]$ satisfies Darboux's equation on $\R^3$:
\begin{equation*}
\partial_r^2M_{x_0}[u] + \frac{2}{r}\partial_rM_{x_0}[u] = \Delta_x M_{x_0}[u],
\end{equation*}
and so by \eqref{eq:wavejohn} $M_{x_0}[u]$ satisfies the $\R^{1+1}$ nonlinear wave equation
\begin{equation*}
-\partial_t^2 (rM_{x_0}[u]) + \partial_r^2 (rM_{x_0}[u]) = -rM_{x_0}[(\partial_t u)^2].
\end{equation*}
For all $t,r \geq 0$ such that $r \geq t$, the Duhamel formula for $M_{x_0}[u]$ is
\begin{equation}\label{eq:duham}
M_{x_0}[u(t)](r) = M_{x_0}[u_{\text{lin}}(t)](r) + \frac{1}{2r}\int_{T(t,r)}\rho M_{x_0}[(\partial_t u(s))^2](\rho)\ud \rho \ud s,
\end{equation}
where $u_{\text{lin}}$ is the solution to the homogeneous equation $\square u_{\text{lin}} = 0$ with the same initial data $u_0,u_1$, and $T(t,r)$ is the triangle with corners $(0,r-t),(t,r),(0,t+r)$. We split the integral of $M[(\partial_t\phi)^2]$ over $T(t,r)$ in \eqref{eq:duham} into parts $\rho \leq r$, and $\rho \geq r$, so that
\begin{multline}\label{eq:triintsplit}
\int_{T(t,r)}\rho M[(\partial_t u(s))^2](\rho)\ud s \ud \rho = \int_{r-t}^r\int_{0}^{\rho-(r-t)} \rho M[(\partial_t u(s))^2](\rho)\ud s \ud \rho \\+ \int_{r}^{r+t}\int_{0}^{r+t-\rho} \rho M[(\partial_t u(s))^2](\rho)\ud s \ud \rho.
\end{multline}
We then differentiate \eqref{eq:duham} with respect to $t$, which by \eqref{eq:triintsplit} gives
\begin{multline}\label{eq:dtduham}
M_{x_0}[\partial_t u(t)](r) = M_{x_0}[\partial_t u_{\text{lin}}(t)](r) + \frac{1}{2r}\int_{r-t}^r\rho M[(\partial_t u(\rho-(r-t)))^2](\rho) \ud \rho\\ + \frac{1}{2r}\int_r^{r+t}\rho M[(\partial_t u(r+t-\rho))^2](\rho) \ud \rho.
\end{multline}
Clearly the second integral in \eqref{eq:dtduham} is nonnegative, moreover we have by Cauchy Schwarz that
\begin{equation*}
(M_{x_0}[f](r))^2 \leq M_{x_0}[f^2](r),
\end{equation*}
and so from \eqref{eq:dtduham} we obtain the lower bound
\begin{equation*}
M_{x_0}[\partial_t u(r-q)](r) \geq M_{x_0}[\partial_t u_{\text{lin}}(r-q)](r) + \frac{1}{2r}\int_{q}^r\rho \Big(M[\partial_t u(\rho-q)](\rho)\Big)^2 \ud \rho,
\end{equation*}
where we set $t = r-q$. Given that the first sign condition \eqref{eq:signconditionforward} is satisfied, we have
\begin{equation}\label{eq:dtduhamnolin}
M_{x_0}[\partial_t u(r-q)](r) \geq \frac{1}{2r}\int_{q}^r\rho \Big(M[\partial_t u(\rho-q)](\rho)\Big)^2 \ud \rho
\end{equation}
for all $r \geq r_0$. Define the function
$$N(r) = \int_{q}^r\rho \Big(M[\partial_t u(\rho-q)](\rho)\Big)^2 \ud \rho.$$
By \eqref{eq:dtduhamnolin}, the following ODE inequality holds for all $r \geq r_0$:
\begin{align}
N'(r) &= r\Big(M[\partial_t u(r-q)](r)\Big)^2\nonumber\\
&\geq r\Big(\frac{1}{2r}\int_{q}^r\rho \Big(M[\partial_t u(\rho-q)](\rho)\Big)^2 \ud \rho\Big)^2\nonumber\\
&= \frac{1}{4r}N^2(r).\label{eq:odeineq}
\end{align}
We then integrate this inequality on the interval $[r_0,r]$, giving
\begin{equation}\label{eq:blowupode}\frac{1}{N(r_0)}-\frac{1}{N(r)} \geq \frac{1}{4}\log\Big(\frac{r}{r_0}\Big).
\end{equation}
Since $u_0,u_1 \in C^2(\R^3)$, the property \eqref{eq:signconditionforward} implies that there exists some small $\epsilon > 0$ such that $M_{x_0}[\partial_t u_{\text{lin}}(r-q)](r) > 0$ for all $r \in [r_0-\epsilon,r_0]$, and so
\begin{align*}
N(r_0) &=  \int_{q}^{r_0}\rho \Big(M[\partial_t u(\rho-q)](\rho)\Big)^2 \ud \rho\\
&\geq \int_{r_0-\epsilon}^{r_0}\rho \Big(M[\partial_t u_{\text{lin}}(\rho-q)](\rho)\Big)^2 \ud \rho > 0.
\end{align*}
It then follows from \eqref{eq:odeineq} that $N(r) \geq N(r_0) > 0$, and so we rearrange \eqref{eq:blowupode} to obtain
\begin{equation}
N(r) \geq \frac{N(r_0)}{N(r_0)-\frac{1}{4}\log(r/r_0)},
\end{equation}
Therefore $N(r) \rightarrow +\infty$ for some $r \leq r_*$, where $r_*$ is the value for which the right hand side blows up to infinity, given by
\begin{equation}
r_* = r_0\exp\Big(\frac{4}{N(r_0)}\Big).
\end{equation}
By \eqref{eq:dtduham}, this implies that the spherical mean $M_{x_0}[\partial_t u(r-q)](r)$ blows up, and hence $\partial_t u \rightarrow \infty$ at some point $(t,x) \in (0,\infty)\times \R^3$.

If alternatively \eqref{eq:signconditionbackward} is satisfied, then observe that the function $v(t) = u(-t)$ also satisfies the semilinear wave equation \eqref{eq:wavejohn} with initial data
$$v(0,x) = u_0(x),\quad \partial_t v(0,x) = -u_1(x).$$
Let $v_{\text{lin}}(t)$ denote the solution to the homogeneous wave equation with the same initial data as $v$. By uniqueness, we have for all $t\in \R$ that $v_{\text{lin}}(t) = u_{\text{lin}}(-t)$. As $\partial_t v_{\text{lin}}(t) = -\partial_t u_{\text{lin}}(-t)$, it is clear that if $u_{\text{lin}}$ satisfies the second sign condition, then $v_{\text{lin}}$ satisfies the first sign condition. From the previous argument, $\partial_t v \rightarrow \infty$ in finite positive time , and so since $\partial_t u(t) = -\partial_t v(-t)$, it follows that $\partial_t u \rightarrow -\infty$ in finite negative time.
    
\end{proof}

Finally we show that the solutions constructed in \cite{yu2022nontrivial} do not satisfy the third sign condition \eqref{eq:signcondition3}, and therefore blow up in finite negative time.
\begin{crl}\label{crl:thm2solblowup}
Let $u$ be a $C^2$ solution to \eqref{eq:wavejohn} constructed in Theorem \ref{thm:yu}, and let $u_{\text{lin}}$ be the solution to the homogeneous wave equation $\square u_{\text{lin}} = 0$ with the same initial data as $u$. Then $u_{\text{lin}}$ does not satisfy the third sign condition \eqref{eq:signcondition3}. That is, there exists some $t \in \R, x \in \R^3$ such that $u_{\text{lin}}(t,x)<0$. Moreoever, $\partial_t u$ blows up in finite negative time.
\end{crl}

\begin{proof}
     We begin by stating that if $u_{\text{lin}}$ does not satisfy \eqref{eq:signcondition3}, then by Lemma \ref{lem:signconditions}, $u_{\text{lin}}$ satisfies either of the first two sign conditions \eqref{eq:signconditionforward},\eqref{eq:signconditionbackward}. It follows by Theorem \ref{thm:bernhardt} that $u$ blows up in finite negative time, since $u$ is a future global solution. To show $u_{\text{lin}}$ does not satisfy the third sign condition, we proceed by contradiction. Suppose $u_{\text{lin}}(t,x) \geq 0$ for all $(t,x)\in\R^{1+3}$. One immediate consequence is that $u_0 \geq 0$. For $t > 0$, by Kirchhoff's formula, we write
\begin{equation}\label{eq:kirchhoffulin}
u_{\text{lin}}(t,x) = \partial_r(rM_{x}[u_0])(t) + tM_{x}[u_1](t) > 0.
\end{equation}
Rearranging this we can bound
$$\partial_r(rM_{x}[u_0])(t)  \geq - tM_{x}[u_1](t),$$
Similarly,
$$u_{\text{lin}}(-t,x) = \partial_r(rM_{x}[u_0])(t) - tM_{x}[u_1](t),$$
and therefore we have
\begin{equation}\label{eq:53}
\partial_r(rM_{x}[u_0])(t) \geq |tM_{x}[u_1](t)|.
\end{equation}
The d'Alembert formula for the spherical means of $u_{\text{lin}}$ is
\begin{multline*}
M_{x}[u_{\text{lin}}(t)](r) = \frac{1}{2r}\Big((r+t)M_{x}[u_0](r+t) +(r-t)M_{x}[u_0](r-t)\Big)\\
+ \frac{1}{2r}\int_{r-t}^{r+t}\rho M_{x}[u_1](\rho)\ud\rho.
\end{multline*}
We bound the final integral using \eqref{eq:53}, so that
\begin{align*}
\int_{r-t}^{r+t}\rho M_{x}[u_1](\rho)\ud\rho &\geq -\int_{r-t}^{r+t}|\rho M_{x}[u_1](\rho)|\ud\rho\\
&\geq -\int_{r-t}^{r+t}\partial_\rho(\rho M_{x}[u_0])(\rho)\ud\rho\\
&=(r-t)M_{x_0}[u_0](r-t) - (r+t)M_{x}[u_0](r+t).
\end{align*}
Hence, by \eqref{eq:53} we have
$$M_{x}[u_{\text{lin}}(t)](r) \geq \frac{(r-t)M_{x}[u_0](r-t)}{r}.$$
Note that $u_0$ must be nonzero somewhere, otherwise by the representation formula, \eqref{eq:kirchhoffulin}, $u_{\text{lin}}$ must have a point $(t,x)\in\R^{1+3}$ where $u_{\text{lin}}(t,x) < 0 < u_{\text{lin}}(-t,x)$. Therefore, we may fix $x_0 \in \R^3$, $q = r-t \geq 0$ such that $M_{x_0}[u_0](q) > 0$, and so for all $r \geq q$
$$M_{x_0}[u_{\text{lin}}(r-q)](r) \geq \frac{C}{r}$$
for some constant $C > 0$. By the Duhamel formula for $u$ \eqref{eq:duham}, we have $u \geq u_\text{lin}$ for all positive time, and so
\begin{equation}\label{eq:sol:decaylower}
M_{x_0}[u(r-q)](r) \geq \frac{C}{r}
\end{equation}

On the other hand, for the solutions $u$ constructed in Theorem \ref{thm:yu}, for all $q \geq 0$, $\omega \in \mathbb{S}^2$, $r$ sufficiently large we may bound $u$ like
\begin{align}
|u(r-q,r\omega)| &\lesssim \epsilon r^{-1}U(\epsilon\ln (r-q) - \delta,q,\omega)+(1+r-q)^{-3/2+C\epsilon}\nonumber\\
&\lesssim \frac{1}{r\ln(r-q)}.\label{eq:sol:decayupper}
\end{align}
These upper and lower bounds cannot both be true, hence $u_{\text{lin}}(t,x) < 0$ for some $(t,x) \in \R^{1+3}$.
\end{proof}

\rmk{\rm The solutions to \eqref{eq:wavejohn} that are constructed in Theorem \ref{thm:yu} are a subset of the more general solutions constructed in \cite{yu2022nontrivial}. In Theorem \ref{thm:yu} it is assumed that for solutions $U$ to the asymptotic system \eqref{eqn:hormasy}, the initial data $U_q(0,q,\omega) = A(q,\omega)$ is compactly supported in $q$. More generally one may assume decay like $A = O(\lra{q}^{-2-})$ when $q<0$ and $A = O(\lra{q}^{-1-})$ when $q>0$, as was shown in \cite{yu2022nontrivial}. In this case, we have (assuming $A =O(\lra{q}^{-\gamma})$ with $\gamma>1$) that
\begin{multline}
|U(q,\omega)|\leq\int_{-\infty}^{\infty} \frac{2A(p,\omega)}{A(p,\omega)s+2}\ \ud p\\
\leq \int_{|p|\leq s^{1/\gamma}}\min\{\frac{1}{s},1\}\ \ud p+\int_{|p|\geq s^{1/\gamma}}\lra{p}^{-\gamma}\ \ud p\lesssim \frac{1}{(1+s)^{1-1/\gamma}}.
\end{multline}
Then in lieu of \eqref{eq:sol:decayupper} we have, for large $r$,
$$|u(r-q,r\omega)| \lesssim \frac{1}{r(\ln(r-q))^{1-1/\gamma}}.$$
As $\gamma > 1$, $u$ still decays faster than $O(r^{-1})$, which is incompatible with the lower bound \eqref{eq:sol:decaylower}. Hence the solutions constructed in \cite{yu2022nontrivial} also blow up in finite negative time.}\rm

\bibliography{main}{}
\bibliographystyle{spmpsci}

\end{document}